\documentclass{amsart}
\usepackage{amssymb, latexsym, amsmath, amsthm}
\usepackage{float}
\usepackage{bbm}
\usepackage{longtable}

\newcommand{\C}{\mathbbm{C}}
\newcommand{\R}{\mathbbm{R}}

\newcommand{\Q}{\mathbbm{Q}}
\newcommand{\Z}{\mathbbm{Z}}
\newcommand{\A}{\mathbb{A}}
\newcommand{\F}{\mathbbm{F}}
\newcommand{\GSp}{{\rm GSp}}
\newcommand{\Sp}{{\rm Sp}}
\newcommand{\GL}{{\rm GL}}

\newcommand{\St}{{\rm St}}

\theoremstyle{plain}
\newtheorem{thm}{Theorem}[section]

\newtheorem{cor}[thm]{Corollary}

\theoremstyle{definition}

\theoremstyle{remark}

\usepackage{fancyhdr}

\usepackage{paralist}
\usepackage{hyperref}

\begin{document}

\title{Dimensions of spaces of Siegel cusp forms of degree 2}
\author{Jeffery Breeding II}
\address{Department of Mathematics, Fordham University, Bronx, NY 10458}
\email{jbreeding@fordham.edu}
\thanks{}

\date{\today}

\maketitle

\begin{abstract}
We give a summary of results for dimensions of spaces of cuspidal Siegel modular forms of degree 2. These results together with a list of dimensions of the irreducible representations of the finite groups $\GSp(4,\F_p)$ are then used to produce bounds for dimensions of spaces of newforms with respect to principal congruence subgroups of odd square-free level.
\end{abstract}

\tableofcontents


\section{Introduction}

The classical theory of the passage of modular forms to automorphic representations is the starting point for its extension to Siegel modular forms of higher degree. The $\GL(1)$ case is famously described in Tate's thesis \cite{Tate}. Let $\A$ denote the ad\`{e}les of the rational numbers $\Q$. Let $\chi_N$ be a Dirichlet character mod $N$. This character can be associated to a continuous character $$\omega: \GL(1,\Q) \backslash \GL(1,\A)\longrightarrow\C^\times,$$ which can be written in terms of local components $$\omega=\otimes_v\omega_v,$$ where $\omega_v$ is a character of $\GL(1,\Q_v)$. The global $L$-function of $\omega$ is constructed from an analysis of the local components $\omega_v$. Many Dirichlet characters are associated to a single such character $\omega$, but among them is a unique primitive one.

Similarly, classical modular forms $f\in \mathcal{S}^1_k(\Gamma(N))$ can be associated to automorphic representations $\pi$ of $\GL(2,\A)$, which can be written in terms of local components $$\pi=\otimes_v\pi_v,$$ where $\pi_v$ is a representation of $\GL(2,\Q_v)$. $\pi$ can be realized in the action of $\GL(2,\A)$ by right translation on a certain space of functions on $\GL(2,\Q)\backslash\GL(2,\A)$. Again, many modular forms are associated to a single such representation $\pi$, but among them is a unique primitive $f$ known as a {\em newform}.

This passage has been studied in great detail. The seminal work for the $\GL(2)$ theory is the book of Jacquet and Langlands \cite{JL}. A survey of the passage has been written by Kudla \cite{Kud}. The reader is also encouraged to consult the works of Bump \cite{Bump}, Diamond and Im \cite{DiaIm}, and Gelbart \cite{Gelb}.

In the higher degree case, we can also associate an eigenform $f$ of degree $g$ to an automorphic representation of ${\rm GSp}(2g,\mathbb{A})=G(\mathbb{A})$. In particular, a useful approach to finding dimensions of spaces of Siegel cusp forms of degree 2 is to investigate the representation theory of $\GSp(4)$. Results in this group's representation theory can be translated to results on spaces of cusp forms. Again, many cusp forms are associated to a single such representation $\pi$, but among them is a unique primitive form known as a {\em newform}. These cuspidal automorphic representations $\pi$ can be written in terms of local components $\pi_v$, where $v$ is a place of $\Q$. The local components of the automorphic representation in turn give rise to local components of the cusp form. One can find the dimensions of certain spaces where some of these local components live. The dimensions tell us essentially how many choices we have for the local factors of the representation and therefore the number of choices of local vectors. The number of associated automorphic representations is then the same as the dimension of the space of newforms.

The paper is organized as follows. After definitions and notations are established, we give dimension formulas for certain spaces of cuspidal Siegel modular forms that already exist in the literature. For automorphic representations associated to cuspidal Siegel modular forms with respect to principal congruence subgroups, we argue a descent to the representations of the finite group $\GSp(4,\F_p)$. The dimensions of the irreducible representations of this finite group were computed by the author \cite{JB}. These dimensions are then used to find bounds for dimensions of certain newforms using existing results on dimensions of spaces of fixed vectors in the local components of certain representations associated with Siegel cusp forms. 

The author wishes to thank his advisor Ralf Schmidt for his guidance and Alan Roche for helpful notes on the supercuspidal case in the theorem below.

\section{Definitions and notations}

Define the general symplectic group $\GSp(4)$ as
\begin{center}
$\GSp(4):=\{g\in\GL(4): {}^tgJg=\lambda J\},$ where  $J=
\begin{pmatrix}
&&&1\\
&&1&\\
&-1&&\\
-1&&&
\end{pmatrix}$
\end{center}
for some $\lambda\neq0$, which will be denoted by $\lambda(g)$ and called the \emph{multiplier} of $g$. The set of all $g\in\GSp(4)$ such that $\lambda(g)=1$ is the subgroup $\Sp(4)$. Any $g\in G$ can be uniquely written as
\begin{displaymath}
g=\begin{pmatrix}
1 & & & \\
& 1 & & \\
& & \lambda(g) & \\
& & & \lambda(g) \\
\end{pmatrix}\,\cdot\, g',
\end{displaymath}
with $g'\in\Sp(4)$.

The {\em Siegel upper half plane of degree 2} is the set of $2\times2$ symmetric matrices with complex entries that have a positive definite imaginary part, i.e.,
$$\mathcal{H}_2=\{Z\in M_2(\C) : {}^tZ=Z, {\rm Im}(Z)>0\}.$$
The symplectic group $\Sp(4,\R)$ acts on $\mathcal{H}_2$ by
$$\begin{pmatrix} A&B\\ C&D\\ \end{pmatrix}\cdot Z=(AZ+B)(CZ+D)^{-1}.$$

Let $\Gamma'$ be a discrete subgroup of $\Sp(4,\R)$. A {\em Siegel cusp form of weight k} is a holomorphic function $f:\mathcal{H}_2\to\C$ such that for all $\gamma=\begin{pmatrix} A&B\\ C&D\\ \end{pmatrix}\in\Gamma'$ and $Z\in\mathcal{H}_2$,
$$f(\gamma\cdot Z)={\rm det}(CZ+D)^k f(Z)$$ with Fourier series expansion $$f(Z)=\sum_{S>0} a_S e^{2\pi i \langle S,Z\rangle},$$
where the series ranges over integral-valued half-integral positive definite $2\times2$ matrices $S$ and $\langle S,Z\rangle= {\rm tr}(SZ)$.

\section{Siegel cusp forms and dimension formulas}

The dimensions of cusp forms of weight $k\geq 4$ have been found for certain subgroups of the modular group $\Gamma=\Sp(4,\Z)$ and we summarize these results here. In fact, Hashimoto \cite{Hash} has given a general formula for ${\rm dim}\, \mathcal{S}_k^2(\Gamma')$ using the Selberg Trace Formula, although it is not explicit. Dimensions of spaces of cusp forms with respect to the following subgroups have been determined using various methods and we summarize those results below.
\begin{itemize}
\item $\Gamma = \Sp(4,\Z)$.
\item $\Gamma_0(N) = \{ g\in\Sp(4,\Z) : g\equiv \begin{pmatrix} A&B\\ & D\\ \end{pmatrix} ({\rm mod}\, N)\}$.
\item $K(N) = \Sp(4,\Q)\cap\left\{\begin{pmatrix}
\Z & \Z & \Z & N^{-1}\Z\\
N\Z & \Z & \Z & \Z\\
N\Z & \Z & \Z & \Z\\
N\Z & N\Z & N\Z & \Z\\
\end{pmatrix}\right\}$.
\item $\Gamma(N) = \{g\in\Sp(4,\Z) : g\equiv I ({\rm mod}\, N)\}$.
\end{itemize}

\subsection{Dimensions of $\mathcal{S}_k(\Gamma)$} 

A formula for the dimension of $\mathcal{S}_k(\Gamma)$ was computed by Eie \cite{Eie} using the Selberg trace formula:
\begin{displaymath}
{\rm dim}\, \mathcal{S}_k(\Gamma)=C(k,2)\int_\mathcal{F} \sum_M K_M(Z,\overline{Z})^{-k}({\rm det} Y)^{k-3}dXdY,
\end{displaymath}
where
$$C(k,2)=2^{-2}(2\pi)^{-3}\left(\prod_{i=0}^1 \Gamma\left(k-\dfrac{1-i}{2}\right)\right)\left( \prod_{i=0}^1 \Gamma\left(k-2+\dfrac{i}{2}\right)\right)^{-1}$$
and $\mathcal{F}$ is a fundamental domain on $\mathcal{H}_2$ for $\Sp(4,\Z)$.

Eie then finds the following dimension formula by determining the contribution from the conjugacy classes of regular elliptic elements in $\Sp(4,\Z)$ using Weyl's character formula for representations of $\GL(2,\C)$, obtaining

$${\rm dim}\, \mathcal{S}_k(\Gamma)=N_1+N_2+N_3+N_4,$$ where
\vspace{0.1in}

$N_1 = \left\{
\begin{array}{ll}
2^{-7}\cdot 3^{-3}\cdot & [1131,229,-229,-1131,427,-571,123,-203,203,\\
&-123,571,-427] \\
&{\rm for}\, k\equiv [0,1,2,3,4,5,6,7,8,9,10,11]\, {\rm (mod\, 12)}
\end{array}
\right.$
\vspace{0.1in}

$N_2=\left\{
\begin{array}{ll}
5^{-1} & {\rm for}\, k\equiv 0\, {\rm (mod\, 5)},\\
-5^{-1} & {\rm for}\, k\equiv 3\, {\rm (mod\, 5)},\\
0 & {\rm otherwise}
\end{array}
\right.$
\vspace{0.1in}

$N_3 = \left\{
\begin{array}{ll}
2^{-5}\cdot 3^{-3}\cdot & [17k-294,-25k+325,-25k+254,17k-261,17k-86,\\
&-k+53,-k-42,-7k+91,-7k+2,-k-27,-k+166,\\
&17k-181] \\
& {\rm for}\, k\equiv [0,1,2,3,4,5,6,7,8,9,10,11]\, {\rm (mod\, 12)}
\end{array}
\right.$
\vspace{0.1in}

$N_4 = \left\{
\begin{array}{ll}
2^{-7}\cdot 3^{-3}\cdot 5^{-1}\cdot (2k^3+96k^2-52k-3231) & {\rm for}\, k\, {\rm even}\\
2^{-7}\cdot 3^{-3}\cdot 5^{-1}\cdot (2k^3-114k^2+2018k-9051) & {\rm for}\, k\, {\rm odd.}
\end{array}
\right.$
\vspace{0.1in}

$\mathcal{S}_k(\Gamma)$ is zero--dimensional for $k<10$. The values for $k=10,11,...,20$ are given in the following table.
\begin{center}
\renewcommand{\arraystretch}{1.5}
\begin{tabular}{|c||c|c|c|c|c|c|c|c|c|c|c|}
\hline
$k$ & 10 & 11 & 12 & 13 & 14 & 15 & 16 & 17 & 18 & 19 & 20\\
\hline
\hline
${\rm dim}\, \mathcal{S}_k(\Gamma)$ & 1 & 0 & 1 & 0 & 1 & 0 & 2 & 0 & 2 & 0 & 3\\
\hline
\end{tabular}
\end{center}

\subsection{Dimensions of $\mathcal{S}_k(\Gamma_0(p))$}

Dimensions of $\mathcal{S}_k(\Gamma_0(p))$ have been computed by Hashimoto \cite{Hash} for weights $k\geq 5$ and by Poor and Yuen \cite{PYGamma0} for weight $k=4$. Poor and Yuen use vanishing theorems and a restriction technique to compute dimensions of $\mathcal{S}_k(\Gamma_0(p))$ for $k=2, 3, 4$ and small primes $p$. For weight 1, Ibukiyama and Skoruppa \cite{IbSk} have shown that $\mathcal{S}_1(\Gamma_0(N))=\{0\}$ for all positive integers $N$.

For weight 4 and small primes $p\leq13$, the dimensions of $\mathcal{S}_4(\Gamma_0(p))$ are

\begin{center}
\renewcommand{\arraystretch}{1.5}
\begin{tabular}{|c||c|c|c|c|c|c|}
\hline
$p$ & 2 & 3 & 5 & 7 & 11 & 13\\
\hline
\hline
${\rm dim}\, \mathcal{S}_4(\Gamma_0(p))$ & 0 & 1 & 1 & 3 & 7 & 11\\
\hline
\end{tabular}
\end{center}

\subsection{Dimensions of $\mathcal{S}_k(K(p))$}

In \cite{PYPara}, Poor and Yuen discuss Siegel modular cusp forms of weight two for the paramodular group $K(p)$ for primes $p<600$ and give a table of dimensions for weight 4 paramodular forms of small prime level. Ibukiyama computed the dimensions for weights $k\geq 5$ using the Selberg trace formula \cite{IbRel} and for weights 3, 4 \cite{IbDim}. In the case of weight $4$, prime level $p\geq 5$ paramodular cusp forms, Ibukiyama determined the dimension formula
$${\rm dim}\, \mathcal{S}_4(K(p))=\frac{p^2}{576}+\frac{p}{8}-\frac{143}{576}+\left(\frac{p}{96}-\frac{1}{8} \right)\left(\frac{-1}{p} \right)+\frac{1}{8}\left(\frac{2}{p}\right)+\frac{1}{12}\left(\frac{3}{p}\right)+\frac{p}{36}\left(\frac{-3}{p} \right).$$

For weight 4 paramodular cusp forms, the dimensions for small prime level spaces are given in the following table.

\begin{center}
\renewcommand{\arraystretch}{1.5}
\begin{tabular}{|c||c|c|c|c|c|c|c|c|}
\hline
$p$ & 2 & 3 & 5 & 7 & 11 & 13 & 17 & 19\\
\hline
\hline
${\rm dim}\, \mathcal{S}_4(K(p))$ & 0 & 0 & 0 & 1 & 1 & 2 & 2 & 3\\
\hline
\end{tabular}
\end{center}

\subsection{Dimensions of $\mathcal{S}_k(\Gamma(p))$}

Dimensions of Siegel cusp forms of degree 2 on the principal congruence subgroup $\Gamma(p)$ have been computed by several people, see \cite{Mor}, \cite{Tsu}, \cite{Yam}. Let $N=p_1\dots p_n$, where $p_1<\dots<p_n$ are distinct odd primes, and let 
$$M=\prod_{i=1}^n(1-p_i^{-2})(1-p_i^{-4}).$$ Then the dimension of Siegel cusp forms of degree 2, weight $k\geq 4$, and level $N$ is
\vspace{0.1in}

${\rm dim}\, \mathcal{S}_k(\Gamma(N))=$
$$N^7 2^{-5}3^{-1}\left(N^32^{-5}3^{-2}5^{-1}(2k-2)(2k-3)(2k-4)-N\cdot2^{-1}3^{-1}(2k-3)+1\right)\cdot M.$$

In particular, if $N=p$ is prime, the dimension is 

\begin{displaymath}
{\rm dim}\, \mathcal{S}_k(\Gamma(p))=\dfrac{\left((2k^3-9k^2+13k-6)p^3+(180-120k)p+360\right)p(p^4-1)(p^2-1)}{2^8 3^3 5}
\end{displaymath}

For convenience, we compute the dimensions of some of these spaces with this formula.
\begin{center}
\renewcommand{\arraystretch}{1.5}
\begin{tabular}{|c||c|c|c|c|c|c|c|c|}
\hline
$p$ & 2 & 3 & 5 & 7 & 11 & 13 & 17\\
\hline
\hline
${\rm dim}\, \mathcal{S}_4(\Gamma(p))$ & 0 & 15 & 5655 & 199500 & 20683575 & 112567455 & 1687834800\\
\hline
\end{tabular}
\end{center}

\begin{center}
\renewcommand{\arraystretch}{1.5}
\begin{tabular}{|c||c|c|c|c|c|c|c|c|}
\hline
$k$ & 4 & 5 & 6 & 7 & 8 & 9 & 10\\
\hline
\hline
${\rm dim}\, \mathcal{S}_k(\Gamma(3))$ & 15 & 76 & 200 & 405 & 709 & 1130 & 1686\\
\hline
\end{tabular}
\end{center}

\begin{center}
\renewcommand{\arraystretch}{1.5}
\begin{tabular}{|c||c|c|c|c|c|c|c|c|}
\hline
$k$ & 4 & 5 & 6 & 7 & 8 & 9 & 10\\
\hline
\hline
${\rm dim}\, \mathcal{S}_k(\Gamma(5))$ & 5655 & 18980 & 43680 & 83005 & 140205 & 218530 & 321230\\
\hline
\end{tabular}
\end{center}
We find that the dimension of the space of cusp forms of weight 4 of the smallest odd square-free level is quite large:
$${\rm dim}\, \mathcal{S}_4(\Gamma(15))=69,023,360,250,000,000.$$

\section{Bounds for dimensions of spaces of newforms}

We now give bounds for dimensions of spaces of newforms $\mathcal{S}_k^{new}(\Gamma(p))$. The idea is look at possible dimensions of $\Gamma(p)$-fixed vectors at the local component $\pi_p$ of an associated automorphic representation.

\begin{thm}
The dimension of the space of newforms $\mathcal{S}_k^{new}(\Gamma(p))$ of weight\\ $k\geq 4$, odd prime level $p$ is bounded below by 

$$\frac{\left((2k^3-9k^2+13k-6)p^3+(180-120k)p+360\right)p(p-1)^2}{34560}$$

and bounded above by

\begin{equation*}
 \begin{cases} \dfrac{6k^3-27k^2-k+82}{12} & \text{if $p=3$,}
\\
\\
\dfrac{\left((2k^3-9k^2+13k-6)p^3+(180-120k)p+360\right)p(p^4-1)}{17280} &\text{if $p\neq3$.}
\end{cases}
\end{equation*}
\end{thm}

\begin{proof}
The case where the local component is non-supercuspidal is discussed in \cite{JB} and its argument is omitted here. In the case where the local component is supercuspidal, we use the work of Morris \cite{Morris1}, \cite{Morris2} and of Moy and Prasad \cite{MoyPrasad}.

Let $\Gamma>\Gamma_1$ be congruence subgroup of $G=\GSp(4,\Q_p)$. Let $\pi$ be an irreducible smooth supercuspidal representation of $G$. Suppose $\pi^{\Gamma_1}\neq 0$. Then $\pi|_\Gamma$ contains a cuspidal representation $\rho$ of $\Gamma/\Gamma_1\cong\GSp(4,\F_p)$. Then $\pi$ contains an extension $\tilde{\rho}$ of $\rho$ to $Z\Gamma$, where $Z$ is the center, and we have ${\rm Hom}_{Z\Gamma}(\tilde{\rho},\pi|_{Z\Gamma})\neq 0$. By Frobenius reciprocity,
$${\rm Hom}_G({\rm ind}_{Z\Gamma}^G \tilde{\rho},\pi)\cong {\rm Hom}_{Z\Gamma}(\tilde{\rho},\pi|_{Z\Gamma}).$$
Since ${\rm ind}_{Z\Gamma}^G\tilde{\rho}$ is irreducible, it must be isomorphic to $\pi$.


Now consider the decomposition of ${\rm ind}_{Z\Gamma}^G\tilde{\rho}|_{\Gamma_1}$ using Mackey's restriction formula, 
$${\rm ind}_{Z\Gamma}^G\tilde{\rho}|_{\Gamma_1}\cong\bigoplus_{x\in Z\Gamma\backslash G/\Gamma_1} {\rm ind}_{Z\Gamma^x\cap\Gamma_1}^{\Gamma_1} (\tilde{\rho}^x|_{Z\Gamma^x\cap\Gamma_1}).$$
where $Z\Gamma^x=x^{-1}Z\Gamma x$, and $\tilde{\rho}^x$ is the representation of $Z\Gamma^x$ defined by $$\tilde{\rho}^x(x^{-1}hx) = \tilde{\rho}(h)$$ for $h\in Z\Gamma.$ We now want to find when ${\rm ind}_{Z\Gamma^x\cap\Gamma_1}^{\Gamma_1} (\tilde{\rho}^x|_{Z\Gamma^x\cap\Gamma_1})$ contains the trivial representation of $\Gamma_1$, i.e, when 
$${\rm Hom}_{\Gamma_1}({\rm ind}_{Z\Gamma^x\cap\Gamma_1}^{\Gamma_1}(\tilde{\rho}^x), \textbf{1}_{\Gamma_1})\neq 0$$
or, equivalently, when 
$${\rm Hom}_{Z\Gamma^x\cap\Gamma_1}(\tilde{\rho}^x|_{Z\Gamma^x\cap\Gamma_1},\textbf{1}_{Z\Gamma^x\cap\Gamma_1})\neq 0.$$
Consider ${\rm ind}_{Z\Gamma}^G\tilde{\rho}|_{\Gamma_1}$. If this contains the trivial representation, then ${\rm ind}_{Z\Gamma}^G\tilde{\rho}|_\Gamma$ contains an irreducible representation, say $\tau$, such that $\tau|_{\Gamma_1}\supset\textbf{1}_{\Gamma_1}$.
\vspace{0.1in}
This implies that $\tau$ is trivial on $\Gamma_1$. So $\pi$ contains $\rho$ and $\tau$. The general theory implies that 
$\rho,\tau$ intertwine, i.e., there exist $x\in G$ such that
$${\rm Hom}_{\Gamma^x\cap\Gamma}(\rho^x,\tau)\neq 0.$$
This implies $x\in Z\Gamma$. So, by the cuspidality of $\rho$ and using representatives for $\Gamma\backslash G/\Gamma$, we have $\rho\cong\tau$. 
It follows that $\pi^{\Gamma_1}=\tilde{\rho}|_{\Gamma_1}$. In particular,
$${\rm dim}\, \pi^{\Gamma_1}={\rm dim}\, \tilde{\rho}={\rm dim}\, \rho.$$
Thus, by considering spaces of $\Gamma(p)$-fixed vectors, the dimensions of the nontrivial irreducible representations of $\GSp(4,\F_p)$ can be used to find bounds for the number of associated automorphic representations, i.e, to find bounds for the dimension of the space of newforms.

The dimensions of the nontrivial irreducible representations of $\GSp(4,\F_q)$, determined in \cite{JB}, are given in the following table.
\begin{center}
\renewcommand{\arraystretch}{1.3}
\begin{longtable}{|l|l||l|l|}
\caption[Dimensions of non-trivial irreducible representations of $\GSp(4,F_p)$]{Dimensions of non-trivial irreducible representations of $\GSp(4,F_p)$}\label{GSp4FqIrreps} \\

\hline
   \multicolumn{1}{|l|}{\strut\centering\small\textbf{Notation}\strut} &
   \multicolumn{1}{|l||}{\strut\centering\small\textbf{Dimension}\strut} &
   \multicolumn{1}{|l|}{\strut\centering\small\textbf{Notation}\strut} &
   \multicolumn{1}{|l|}{\strut\centering\small\textbf{Dimension}\strut} \\
\hline
\hline
\endfirsthead

\multicolumn{4}{c}{{\tablename} \thetable{} -- Continued} \\[0.5ex]
\hline
   \multicolumn{1}{|l|}{\strut\centering\small\textbf{Notation}\strut} &
   \multicolumn{1}{|l||}{\strut\centering\small\textbf{Dimension}\strut} &
   \multicolumn{1}{|l|}{\strut\centering\small\textbf{Notation}\strut} &
   \multicolumn{1}{|l|}{\strut\centering\small\textbf{Dimension}\strut} \\
\hline
\hline
\endhead

 \hline
\endlastfoot


$a_1(p)$ & $(p^2+1)(p+1)^2$ & $a_{10}(p)$ & $(p^2+1)(p+1)$\\
\hline
$a_2(p)$ & $p(p^2+1)(p+1)$ & $a_{11}(p)$ & $p(p^2+1)$\\
\hline
$a_3(p)$ & $p^2(p^2+1)$ & $a_{12}(p)$ & $(p^2+1)(p-1)$\\
\hline
$a_4(p)$ & $p^4$ & $a_{13}(p)$ & $\frac{1}{2}p(p+1)^2$\\
\hline
$a_5(p)$ & $p^4-1$ & $a_{14}(p)$ & $\frac{1}{2}p(p^2+1)$\\
\hline
$a_6(p)$ & $p^2(p^2-1)$ & $a_{15}(p)$ & $\frac{1}{2}p(p-1)^2$\\
\hline
$a_7(p)$ & $(p^2-1)^2$ & $a_{16}(p)$ & $p^2+1$\\
\hline
$a_8(p)$ & $p(p^2+1)(p-1)$ & $a_{17}(p)$ & $p^2-1$\\
\hline
$a_9(p)$ & $(p^2+1)(p-1)^2$ & &\\
\end{longtable}
\end{center}
We may exclude the dimensions $a_{16}(p)$ and $a_{17}(p)$ from our considerations because they are dimensions of $\Gamma(p)$-fixed vectors of representations that are not unitary.  The theorem follows.
\end{proof}

We note that this method also applies to finding bounds for dimensions of Siegel cusp forms of higher degree $g$. However, one must determine the dimensions of the irreducible representations of the group $\GSp(2g,\F_p)$ in the higher degree case.

\begin{cor}
The dimension of the space of newforms of weight 4, level 3 is 1. Moreover, 
\begin{itemize}
\item $\mathcal{S}_4^{new}(\Gamma(3))=\mathcal{S}_4(\Gamma_0(3))$.
\item All Siegel cusp forms of weight 4, level 3 are Saito-Kurokawa lifts.
\item The associated automorphic representation's component at $p=3$ is isomorphic to the non-supercuspidal representation $\tau(T,\nu^{-1/2}\sigma)$, a constituent of $\nu\times1_{\Q_p^\times}\rtimes\nu^{-1/2}\sigma$, where $\nu$ is the valuation of $\Q_p$.
\end{itemize}
\end{cor}
\begin{proof}
First note that ${\rm dim}\, \mathcal{S}_4(\Gamma(3))=15$. By our theorem, we have
$$\dfrac{3}{32}\leq{\rm dim}\, \mathcal{S}_4^{new}(\Gamma(3))\leq\dfrac{5}{2}.$$
Moreover, we can determine the dimension exactly in this case. Let $f$ be an eigenform of weight 4 and level 3 and let $\pi=\pi_f=\otimes_v \pi_v$ be its associated automorphic representation. Since $f$ has level 3, $\pi_p$ is spherical for finite primes $p\neq3$ and $\pi_3$ has a non-trivial finite-dimensional subspace of $\Gamma(3)$-fixed vectors. It is clear that only solution of the Diophantine equation
$$\sum_{n=1}^{15}c_na_n(3)={\rm dim}\, \mathcal{S}_4(\Gamma(3))=15$$
is $c_{14}=1$ and $c_i=0$ for $i\neq14$. This means that there is only one automorphic representation associated to this space. Hence, the dimension of the space of newforms is 1.

Furthermore, the irreducible representations of $\GSp(4,\F_3)$ that have dimension $a_{14}(3)$ are the non-cuspidal representations ${\rm Ind}(\theta_{11})_a, {\rm Ind}(\theta_{12})_b$, and their twists, see \cite{JB}. These representations descend from the non-supercuspidal representations $\tau(T,\nu^{-1/2}\sigma)$ or $L(\nu^{1/2}\St_{\GL(2)},\nu^{-1/2}\sigma)$. To determine which one it is, we note that $\tau(T,\nu^{-1/2}\sigma)$ has a non-zero $\Gamma_0(3)$-fixed vector and $L(\nu^{1/2}\St_{\GL(2)},\nu^{-1/2}\sigma)$ does not, see \cite{ST}. Also, $L(\nu^{1/2}\St_{\GL(2)},\nu^{-1/2}\sigma)$ has a non-zero $\Gamma^{K(3)}$-fixed vector and $\tau(T,\nu^{-1/2}\sigma)$ does not. So the correct local component can be identified if we know ${\rm dim}\, \mathcal{S}_4(\Gamma_0(3))$ or ${\rm dim}\, \mathcal{S}_4(K(3))$. From \cite{PYGamma0} and \cite{PYPara} we have ${\rm dim}\, \mathcal{S}_4(\Gamma_0(3))=1$ and\\ ${\rm dim}\, \mathcal{S}_4(K(3))=0$. So $\tau(T,\nu^{-1/2}\sigma)$, a Saito-Kurokawa lifting, is the local component.
\end{proof}

Similarly, we also have bounds for dimensions of newforms of odd square-free level.
\begin{thm}
Let $N=p_1\dots p_n$, where $p_1<\dots<p_n$ are distinct odd primes, and let 
$$M=\prod_{i=1}^n(1-p_i^{-2})(1-p_i^{-4})$$

The dimension of the space of newforms $\mathcal{S}_k^{new}(\Gamma(N))$ of weight\\ $k\geq 4$ and odd square-free level $N=p_1\dots p_n$, where $p_1<\dots<p_n$ are distinct odd primes, is bounded below by 

$$\dfrac{N^7 2^{-5}3^{-1}\left(N^32^{-5}3^{-2}5^{-1}(2k-2)(2k-3)(2k-4)-N\cdot2^{-1}3^{-1}(2k-3)+1\right)}{\sum_{i=1}^n (p_i^2+1)(p_i+1)^2}\cdot M.$$
The dimension is bounded above by
$$\dfrac{N^7 2^{-5}3^{-1}\left(N^32^{-5}3^{-2}5^{-1}(2k-2)(2k-3)(2k-4)-N\cdot2^{-1}3^{-1}(2k-3)+1\right)}{6+\sum_{i=2}^n (p_i^2-1)}\cdot M$$
if $3|N$ or by
$$\dfrac{N^7 2^{-5}3^{-1}\left(N^32^{-5}3^{-2}5^{-1}(2k-2)(2k-3)(2k-4)-N\cdot2^{-1}3^{-1}(2k-3)+1\right)}{\sum_{i=1}^n (p_i^2-1)}\cdot M$$
if $3\nmid N$.
\end{thm}

\bibliographystyle{amsplain}

\begin{thebibliography}{10}

\bibitem{JB}
J. Breeding II, {\em Irreducible non-cuspidal characters of $\GSp(4,\F_q)$}, Ph.D. thesis, University of Oklahoma, Norman, OK, 2011.

\bibitem{Bump}
D. Bump, {\em Automorphic Forms and Representations}, Cambridge University Press, Cambridge, UK, 1997.

\bibitem{DiaIm}
F. Diamond and J. Im, {\em Modular forms and modular curves}, Seminar on Fermat's Last Theorem, pp. 39--133, CMS Conf. Proc., 17, {\em Amer. Math. Soc.}, Providence, RI, 1995.

\bibitem{Eie}
M. Eie, {\em Contributions from conjugacy classes of regular elliptic elements in $\Sp(n,\Z)$ to the dimension formula}, Trans. Amer. Math. Soc. 285 (1984), no. 1, 403--410.

\bibitem{Gelb}
S. Gelbart, {\em Automorphic forms on ad\`{e}le groups}, Annals of Mathematics Studies, No. 83. Princeton University Press, Princeton, N.J.; University of Tokyo Press, Tokyo, 1975.

\bibitem{Hash}
K. Hashimoto, {\em The dimension of the spaces of cusp forms on the Siegel upper half-plane of degree two. I}, J. Fac. Sci. Univ. Tokyo Sect. IA Math. \textbf{30} (1983), no. 2, 403--488.

\bibitem{IbRel}
T. Ibukiyama, {\em On relations of dimensions of automorphic forms of $\Sp(2,\R)$ and its compact twist $\Sp(2)$. I.}, Automorphic forms and number theory (Sendai 1983), 7--30, Adv. Stud. Pure Math., 7, North-Hollana, Amsterdam, 1985.

\bibitem{IbDim}
T. Ibukiyama, {\em Dimensions formulas of Siegel modular forms of weight 3 and supersingular abelian surfaces}, Siegel Modular Forms and Abelian Varieties, Proceedings of the 4th Spring Conference on Modular Forms and Related Topics, 2007, pp. 39--60.

\bibitem{IbSk}
T. Ibukiyama and N.-P. Skoruppa, {\em A vanishing theorem for Siegel modular forms of weight one.}, Abh. Math. Sem. Univ. Hamburg \textbf{77} (2007), 229--235.

\bibitem{JL}
H. Jacquet and R. Langlands, {\em Automorphic forms on GL(2)}, Lecture Notes in Mathematics, Vol. 114, Springer-Verlag, Berlin-New York, 1970.

\bibitem{Kud}
S. Kudla, {\em From modular forms to automorphic representations}, An introduction to the Langlands program (Jerusalem, 2001), pp. 131--151, Birkh\"{a}user Boston, Boston, MA, 2003.

\bibitem{Mor}
Y. Morita, {\em An explicit formula for the dimension of spaces of Siegel modular forms of degree two}, J. Fac. Sci. Univ. Tokyo Sect. IA Math. \textbf{21} (1974), 167--248.

\bibitem{Morris1}
L. Morris, {\em Tamely ramified supercuspidal representations}, Ann. Sci. \'{E}cole Norm. Sup. (4) \textbf{29} (1996), no. 5, 639--667.

\bibitem{Morris2}
L. Morris, {\em Level zero \textbf{G}-types}, Compositio Math. \textbf{118} (1999), no. 2, 135--157.

\bibitem{MoyPrasad}
A. Moy and G. Prasad, {\em Jacquet functors and unrefined minimal \textbf{K}-types}, Comment. Math. Helv. \textbf{71} (1996), no. 1, 98--121.

\bibitem{PYGamma0}
C. Poor and D. Yuen, {\em Dimensions of cusp forms for $\Gamma_0(p)$ in degree two and small weights.}, Abh. Math. Sem. Univ. Hamburg \textbf{77} (2007), 59--80.

\bibitem{PYPara}
C. Poor and D. S.Yuen, {\em Paramodular Cusp Forms}, arXiv:0912.0049v1 (2009).

\bibitem{ST}
P. Sally and M. Tadi\'{c}, {\em Induced representations and classifications for GSp(2, F) and Sp(2, F)}, M\'{e}m. Soc. Math. France (N.S.) No. 52 (1993), 75--133.

\bibitem{Tate}
J. Tate, {\em Fourier analysis in number fields and Hecke's zeta-functions}, Ph.D. thesis, Princeton University, Princeton, NJ, 1950.

\bibitem{Tsu}
R. Tsushima, {\em On the spaces of Siegel cusp forms of degree two}, Amer. J. Math. \textbf{104} (1982), no. 4, 843--885.

\bibitem{Yam}
T. Yamazaki, {\em On Siegel modular forms of degree two}, Amer. J. Math. \textbf{98} (1976), no. 1, 39--53.
\end{thebibliography}

\end{document}